\documentclass[10pt]{amsart}

\usepackage[pagebackref,colorlinks=true]{hyperref}  %za cvetni linkove (obache bavi)

\usepackage{amsfonts}
\usepackage{amsmath,amsthm,amssymb,amscd,enumerate,eucal,url}

\usepackage{eucal,url,amssymb,verbatim,%booktabs,%
enumerate,amscd,%paralist
}

\numberwithin{equation}{section}

\newtheorem{thrm}{Theorem}[section]

\newtheorem{lemma}[thrm]{Lemma}

\newtheorem{cor}[thrm]{Corollary}

\newtheorem{dfn}[thrm]{Definition}

\newtheorem{rmrk}[thrm]{Remark}

\newtheorem{exm}[thrm]{Example}

\newtheorem{conv}[thrm]{Convention}

\usepackage[margin=1in]{geometry}
\newcommand{\R}{\mathbb{R}}

\newcommand{\Hn}{\mathbb{H}^n}
\newcommand{\Hnn}{\mathbb{H}^{n+1}}
\newcommand{\QH}{\boldsymbol {G\,(\mathbb{H})}}

\newcommand{\abs}[1]{\lvert #1 \rvert}

\newcommand{\e}{\textbf {e}}

\newcommand{\A}{{A}}
\newcommand{\B}{{B}}
\newcommand{\C}{{C}}
\newcommand{\D}{{D}}
\newcommand{\spl}{\mathfrak{ sp}(1)}

\newcommand{\dxa}{{\partial_{x_{\alpha}}} }
\newcommand{\dta}{{\partial_{t_{\alpha}}} }
\newcommand{\dya}{{\partial_{y_{\alpha}}} }
\newcommand{\dza}{{\partial_{z_{\alpha}}} }

\newcommand{\dt}{{\partial_t} }
\newcommand{\dx}{{\partial_x}}
\newcommand{\dy}{{\partial_y} }
\newcommand{\dz}{{\partial_z} }

\begin{document}

\begin{abstract}
We describe explicitly all quaternionic contact  hypersurfaces (qc-hypersurfaces) in the flat quaternion space $\Hnn$ and the quaternion projective space. We show that up to a quaternionic affine transformation  a qc-hypersurface in $\Hnn$ is   contained in one of the three qc-hyperquadrics in $\Hnn$.  Moreover, we show that an embedded qc-hypersurface in a hyper-K\"ahler manifold is  qc-conformal to a qc-Einstein space and the Riemannian curvature tensor of the  ambient hyper-K\"ahler metric is degenerate along the hypersurface.
\end{abstract}



\title[Quaternionic contact  hypersurfaces in hyper-K\"ahler manifolds]
{Quaternionic contact  hypersurfaces in hyper-K\"ahler manifolds}

\date{\today}

\author{Stefan Ivanov}
\address[Stefan Ivanov]{University of Sofia, Faculty of Mathematics and Informatics,
blvd. James Bourchier 5, 1164,
Sofia, Bulgaria}
\address{and Institute of Mathematics and Informatics, Bulgarian Academy of
Sciences}
\email{ivanovsp@fmi.uni-sofia.bg}

\author{Ivan Minchev}
\address[Ivan Minchev]{University
of Sofia, Faculty of Mathematics and Informatics, blvd. James Bourchier 5, 1164 Sofia, Bulgaria;
Department of Mathematics and Statistics, Masaryk University, Kotlarska 2, 61137 Brno,
Czech Republic}
\email{minchev@fmi.uni-sofia.bg}

\author{Dimiter Vassilev}
\address[Dimiter Vassilev]{
University of New Mexico\\
%Department of Mathematics\\
Albuquerque, NM 87131}
\email{vassilev@unm.edu}

\maketitle

%\date{April 17, 2006}

\setcounter{tocdepth}{2} \tableofcontents

\section{Introduction}
 A quaternionic contact (abbr. qc) hypersurface  of a quaternionic manifold $(N,\mathcal Q)$ was defined by Duchemin \cite{D1} as a hypersurface $M$ endowed with a qc-structure compatible with the induced quaternion structure on the maximal quaternion invariant subspace $H$ of the tangent space of $M$. It was shown in \cite[Theorem 1.1]{D1}  that a qc manifold can be realized as a qc-hypersurface of an abstract quaternionic manifold. In this paper we investigate qc-hypersurfaces embedded in a hyper-K\"ahler manifold and, in particular,  qc-hypersufaces of the flat quaternion  space  $\R^{4n+4}\cong \Hnn$.

A hypersurface of a hyper-K\"ahler manifold inherits a quaternionic contact structure  from the ambient hyper-K\"ahler structure if the second fundamental form restricted to $H$ is Sp(1)-invariant and definite quadratic tensor, \cite{D1,IMV1}. Considering $\Hnn$ as a flat hyper-K\"ahler manifold, a natural  question is the embedding problem for an abstract qc-manifold.

Our first main result  describes the embedded in $\Hnn$ qc-hypersurfaces.

\begin{thrm}\label{t:embedded in Hn} If $M$ is a connected qc-hypersurface of $\R^{4n+4}\cong \Hnn$ then, up to a quaternionic affine transformation of $\Hnn$, $M$ is contained in one of the following three hyperquadrics:
$$|q_1|^2+\dots+|q_n|^2 +\R{e}(p)=0,\qquad
|q_1|^2+\dots+|q_n|^2 + |p|^2=1,\qquad |q_1|^2+\dots+|q_n|^2 - |p|^2=-1.$$

\noindent Here $(q_1,q_2,\dots q_n,p)$ denote the standard quaternionic coordinates of $\Hnn.$

In particular, if $M$ is a compact qc-hypersurface of $\R^{4n+4}\cong \Hnn$ then, up to a quaternionic affine transformation of $\Hnn$, $M$ is the standard 3-Sasakian sphere.
\end{thrm}

More generally,  considering  qc-hypersurfaces of a hyper-K\"ahler manifold, we show  that any such qc-hypersurface is  qc-conformally  equivalent to a qc-Einstein manifold, i.e., the qc-conformal class of any embedded qc-structure contains one for which the horizontal Ricci tensor of the associated Biquard connection is proportional to the metric on the horizontal bundle. We note that qc-Einstein spaces are locally qc-homothetic to a certain $SO(3)$-bundles over  a quaternionic K\"ahler manifold of either positive scalar curvature (3-Sasakian spaces), or non-positive scalar curvature \cite{IMV1,IV2,IMV4}.

\begin{thrm}\label{t:qc-embedded in hyperkahler}
If $M$ is a qc-manifold embedded as a hypersurface in a
hyper-K\"ahler manifold, then $M$ is qc-conformal to a qc-Einstein structure.
\end{thrm}

We obtain our second main result in the course of the proof of a stronger result, cf. Theorem \ref{t:surface1} and Lemma \ref{l:step1}

We also find necessary conditions for the existence of a qc-hypersurface in a hyper-K\"ahler manifold, namely the Riemannian curvature $R$ of the ambient space has to be degenerate along the normal to the qc-hypersurface vector field, see Theorem~\ref{surface2}.  From this point of view the "richest" ambient space  is the flat space $\mathbb H^{n+1}\cong\R^{4n+1}$ in which case Theorem~\ref{t:embedded in Hn} provides a complete description.

Our approach to the considered problems is partially motivated by  \cite[Corollary B]{L1} who showed that a non-degenerate CR manifold embedded as a hypersurface in $\mathbb{C}^{n+1}$, $n\geq2$, admits a pseudo-Einstein structure, i.e., there is a contact form for which the pseudo-hermitian Ricci tensor of the Tanaka-Webster connection is proportional to the Levi form. A key insight of \cite[Theorem 4.2]{L1} is  that a contact form $\theta$ defines a pseudo-Hermitian structure  which is pseudo-Einstein iff locally there exists a closed section of the canonical bundle with respect to which $\theta$ is volume-normalized. In the considered here quaternionic setting, we show the existence of a "calibrated"  qc-structure which is volume normalizing in a certain sense, see Lemma \ref{mu} and \eqref{def:f1}.

\begin{conv}\label{conven}

Throughout the paper, unless explicitly stated otherwise, we will
use the following notation.

\begin{enumerate}[a)]\label{e:notation}
\item The triple $(i,j,k)$ denotes any cyclic permutation of
$(1,2,3)$.
\item $s,t$ are any numbers from the set $\{1,2,3\}$, $s,t \in \{1,2,3\}$.
\item For a given decomposition $TM=V\oplus H$ we denote by
$[.]_V$ and $[.]_H$ the corresponding projections to $V$ and $H$.
\item ${\A}, {\B}, {\C}$, etc. will denote
sections of the tangent bundle of $M$,  ${\A},
{\B}, {\C}\in  TM$.
\item $X,Y,Z,U$ will denote horizontal vector fields,
$X,Y,Z,U\in H$.
\end{enumerate}
\end{conv}

\section{Preliminaries}
\subsection{QC-manifolds}
It is well known that the sphere at infinity of a  non-compact symmetric space $M$ of rank one carries a natural
Carnot-Carath\'eodory structure, see \cite{M,P}.  Quaternionic contact (qc) structures were introduced by O. Biquard \cite{Biq1} modeling  the conformal boundary at infinity of the quaternionic hyperbolic space. Biquard showed that the infinite dimensional family \cite{LeB91} of complete quaternionic-K\"ahler deformations of the quaternion hyperbolic metric have conformal infinities which provide an infinite dimensional family of examples of qc structures. Conversely, according to \cite{Biq1} every real analytic qc structure on a manifold $M$ of dimension at least eleven is the conformal infinity of a unique quaternionic-K\"ahler metric defined in a neighborhood of $M$.

 We refer to \cite{Biq1}, \cite{IMV1} and \cite{IV3} for a more detailed exposition of the definitions and properties of qc structures and the associated Biquard connection. Here, we recall briefly the relevant facts needed for this paper. A quaternionic contact (qc) manifold is a $4n+3$%
-dimensional manifold  $M$ with a codimension three distribution $H$  equipped with  an $Sp(n)Sp(1)$ structure locally defined by an $\mathbb{R}^3$-valued 1-form $\eta=(\eta_1,\eta_2,\eta_3)$. Thus, $H=\cap_{s=1}^3 Ker\, \eta_s$
carries a positive definite symmetric tensor $g$, called the horizontal metric, and a compatible rank-three bundle $\mathbb{Q}$
consisting of endomorphisms of $H$ locally generated by three orthogonal almost complex
structures $I_s$,  satisfying the unit quaternion relations: (i) $I_1I_2=-I_2I_1=I_3, \quad $ $I_1I_2I_3=-id_{|_H}$; \hskip.1in (ii) $g(I_s.,I_s.)=g(.,.)$; and \hskip.1in  (iii) the
compatibility conditions  $2g(I_sX,Y)\ =\ d\eta_s(X,Y)$, $
X,Y\in H$  hold true.

The transformations preserving a given quaternionic contact
structure $\eta$, i.e., $\bar\eta=\mu\Psi\eta$ for a positive smooth
function $\mu$ and an $SO(3)$ matrix $\Psi$ with smooth functions as
entries are called \emph{quaternionic contact conformal (qc-conformal) transformations}.  The qc-conformal curvature tensor $W^{qc}$, introduced in \cite{IV1}, is the
obstruction for a qc structure to be locally qc-conformal to the
standard 3-Sasakian structure on the $(4n+3)$-dimensional sphere \cite{IV1,IV3}. It is a noteworthy and well known fact that, unlike the CR geometry, in the qc case the horizontal space determines uniquely the qc-conformal class, see Lemma \ref{l:admis-lemma}.

As shown in \cite{Biq1} there is a "canonical" connection associated to every qc manifold of dimension at least eleven. In the seven dimensional case the existence of such a connection requires the qc-structure to be integrable \cite{D}. The integrability condition is equivalent to the existence of Reeb vector fields \cite{D}, which (locally) generate the supplementary to $H$ distribution $V$. The Reeb vector fields $%
\{\xi_1,\xi_2,\xi_3\}$ are determined by \cite{Biq1}
\begin{equation}  \label{bi1}
 \eta_s(\xi_t)=\delta_{st}, \qquad (\xi_s\lrcorner
d\eta_s)_{|H}=0,\quad (\xi_s\lrcorner d\eta_t)_{|H}=-(\xi_t\lrcorner
d\eta_s)_{|H},
\end{equation}
where $\lrcorner$ denotes the interior multiplication.  Henceforth, by a qc structure in dimension $7$ we shall mean a qc structure satisfying \eqref{bi1} and refer to the "canonical" connection as \emph{the Biquard connection}. The Biquard connection is the unique linear connection preserving the decomposition $TM=H\oplus V$ and the $Sp(n)Sp(1)$ structure on $H$ with torsion $T$ determined by $T(X,Y)=-[X,Y]_{|_V}$ while  the endomorphisms $T({\xi_s},.): H \rightarrow H$ belong to the orthogonal complement $(sp(n)+sp(1))^{\perp}\subset GL(4n,R)$.

The covariant derivatives with respect to the Biquard connection of
the endomorphisms $I_s$  and the Reeb vector fields are given by
\begin{equation*}%\label{der}
\nabla I_i=-\alpha_j\otimes I_k+\alpha_k\otimes I_j,\qquad
\nabla\xi_i=-\alpha_j\otimes\xi_k+\alpha_k\otimes\xi_j.
\end{equation*}
 The $\spl$-connection 1-forms  $\alpha_1,\alpha_2, \alpha_3$, defined by the
above equations satisfy \cite{Biq1}
\begin{equation*}%\label{hor-alpha}
\alpha_i(X)=d\eta_k(\xi_j,X)=-d\eta_j(\xi_k,X),\qquad X\in H.
\end{equation*}

 Let $R=[\nabla,\nabla]-\nabla_{[.,.]}$ be the curvature tensor of $\nabla$ and
$R(\A,\B,\C,\D)=g(R_{\A,\B}\C,\D)$ be the corresponding curvature tensor of type (0,4). The qc Ricci tensor $Ric$ and the normalized qc scalar curvature $S$ are defined by
\begin{equation}\nonumber
Ric(\A,\B)=\sum_{a=1}^{4n}R(e_a,\A,\B,e_a)\qquad 8n(n+2)S=Scal=\sum_{a=1}^{4n}Ric(e_a,e_a),
\end{equation}
where $e_1,\dots,\e_{4n}$ of $H$ is an $g$-orthonormal frame on $H$.

We say that $(M,\eta)$ is a qc-Einstein manifold if the restriction of the qc-Ricci tensor to the horizontal space $H$ is  trace-free, i.e.,
$$Ric(X,Y)=\frac{Scal}{4n}g(X,Y)=2(n+2)Sg(X,Y), \quad X,Y\in H.$$
The qc-Einstein condition is equivalent to the vanishing of the torsion endomorphism of the Biquard connection, $T(\xi_s,X)=0$ \cite{IMV1}. It is also known \cite{IMV1,IMV4}  that the qc-scalar curvature of a qc Einstein manifold is  constant.

The structure equations of a qc manifold \cite[Theorem 1.1]{IV2} are given by
\begin{equation}\label{streq}
\begin{aligned}
d\eta_i  =2\omega_i-\eta_j\wedge\alpha_k+\eta_k\wedge\alpha_j -
S
\eta_j\wedge\eta_k,
\end{aligned}
\end{equation}
where $\omega_s$ are the fundamental 2-forms defined by the equations
\begin{equation*}  %\label{thirteen}
2\omega_{s|H}\ =\ \, d\eta_{s|H},\qquad
\xi_t\lrcorner\omega_s=0.
\end{equation*}
 By \cite[Theorem 5.1]{IMV4}, see also \cite{IV2} and \cite[Theorem 4.4.4]{IV3} for alternative proofs in the case $Scal\not=0$, a qc-Einstein structure is characterised by either of the following equivalent conditions
\begin{enumerate}[i)]
%\item $M$ is a qc-Einstein manifold;
\item  locally, the given qc-structure is defined by  1-form $(\eta_1,\eta_2,\eta_3)$  such that for some constant $S$ we have
\begin{equation}\label{str_eq_mod}
d\eta_i=2\omega_i+S\eta_j\wedge\eta_k;
\end{equation}
\item locally,  the given qc-structure is defined by a 1-form $(\eta_1,\eta_2,\eta_3)$  such that the corresponding connection 1-forms vanish on $H$ and (cf. the proof of Lemma 4.18 of \cite{IMV1})
\begin{equation}\label{e:vanishing alphas qc-Einstein}
\alpha_s=-S\eta_s.
\end{equation}
\end{enumerate}

\subsection{QC-hypersurfaces}\label{ss:qc-hypersurfaces}
Let $M$ be an oriented real hypersurface in a hyper-K\"ahler  manifold
$K$ with parallel quaternionic bundle $\mathcal{Q}$ and   maximal $\mathcal{Q}$-invariant subspace $H\subset TK$ called the horizontal space.  The
hypersurface $M$ is a qc-hypersurface of $K$ if it is a qc
manifold with respect to the induced  on $H$ quaternionic structure. We note that the induced qc-structure on $M$ is generated by  globally defined 1-forms $\hat{\eta_s}$ determined by the unit normal $N$ to $M$, see \eqref{e:defn of hat-eta} below.  Formally, we rely on the following definition \cite[Proposition 2.1]{D1}.
\begin{dfn}\label{d:QRhypersurface}
Let $K$ be a hyper-K\"ahler manifold with hyper-complex structure
$(J_1,J_2,J_3)$ and hyper-K\"ahler metric  $G$. Let $(M,H,\mathcal{Q})$ be a qc-manifold, and $\iota :M\rightarrow K$ an embedding. We say that $M$ is a qc-embedded hypersurface of $K$ if  $\iota_*(H)$ is a codimension four subbundle of $TK$ which is a $J_s$-invariant subspace of $TM$ for
$s\in\{1,2,3\}$ and the restrictions of $J_1,J_2,J_3$ to the subspace $\iota_*(H)$ are  elements of the induced by $\mathcal{Q}$  quaternionic structure on $\iota_*(H)$.
\end{dfn}
In order to simplify the notation we will
identify the corresponding points and tensor fields on $M$ with their images
through the map $\iota$ in $K$.

We note explicitly that the above definition determines the conformal class of the given qc structure rather than a particular qc structure inside this conformal class, cf. Lemma \ref{l:admis-lemma}.  An equivalent characterization of a qc-hypersurface $M$ is  that the
restriction of the second fundamental form of $M$ to the horizontal space is a definite
symmetric form, which is invariant with respect to the quaternion structure, see \cite[Proposition~2.1]{D1}.  After choosing the unit normal vector $N$  to $M$ appropriately we can \emph{and will} assume that the second fundamental form of $M$ is negative definite on  the horizontal space.

\begin{rmrk} For practical purposes, it is useful to keep in mind the description through a locally defining function $\rho$,  $M=\rho^{-1}(0)$ with non-vanishing differential $d\rho$. By \cite[Proposition~2.1]{D1}, $M$ is a qc-hypersurface iff   pointwise $D d\rho(X,Y)$ is either positive or negative definite and $\mathcal Q$ invariant quadratic form on the maximal $\mathcal Q$-invariant subspace $H$ of $TM$. Furthermore, instead of the Levi-Civita connection $D$ one can take any  torsion free connection on $K$
preserving the quaternion bundle of $K$.
\end{rmrk}

With $|.|$ denoting the length of a tensor determined by the metric $G$, consider
\begin{equation}\label{e:defn of hat-eta}
\hat\eta_s(\A)= G(J_sN,\A)=\frac {1}{|d\rho|}J_s d\rho,\ \A\in TM,\  \
s=1,2,3,
\end{equation}
so that $H=\cap_{s=1}^3 Ker\, \hat\eta_s$. Since the complex structures $J_s$ are parallel
with respect to the Levi-Civita connection $D$ it follows
\begin{gather}\label{e:d hat eta}
d\hat\eta_s(\A,\B)=(D_\A\hat\eta_s)(\B)-(D_{\B}\hat\eta_s)(\A)=
-G(J_s(D_\A N),\B) + G(J_s(D_{\B}N),\A)=\\\nonumber
=II(\A,
[J_s\B]_{TM})- II(\B,[J_s\A]_{TM}), \qquad
\A,\B\in TM.
\end{gather}
Defining $\hat g(X,Y)=-II(X,Y),\ X,Y\in H$, \eqref{e:d hat eta} yields
$d\hat\eta(X,Y)=2g(I_sX,Y)$, which defines a qc-structure $(M,\hat\eta_s,I_s,\hat g)$  in the qc-conformal class determined by the qc-embedding.

The associated Reeb vector fields $\hat\xi_s$,
fundamental 2-forms $\hat\omega_s$, and $\spl$-connection 1-forms $\hat\alpha_s$ are determined easily as follows. For $\hat r_s=\hat\xi_s - J_sN$, since
$\hat\eta_t(\hat r_s) =0$  we have $\hat r_s\in H$.
Using the equation $d\hat\eta_s(\hat\xi_s,X)=0,\ X\in H$ and
\eqref{e:d hat eta} we obtain
\begin{equation*} %\label{e:hat rs}
2II(\hat r_i, X) =- II(J_i N, X).
\end{equation*}
In addition, we have
\begin{multline*}
\hat\alpha_i (X)=d\hat\eta_k(\hat r_j, X)+d\hat\eta_k( J_j N, X)=2II(\hat r_j, I_k X) +d\hat\eta_k( J_j N, X)\\
                 =2II(\hat r_j, I_k X) + II(J_j N, I_k X)+ II(X,J_i N)\\
                 =- II(J_j N, I_k X) + II(J_j N, I_k X)+ II(X,J_i N)
                 =II(J_i N, X).
                 \end{multline*}
Notice that,  unless the three 1-forms $II(J_sN,.)$ vanish on $H$,  the qc-structure $(\hat\eta_s,I_s,\hat g)$ does not satisfy the structure
equations
$
d\hat\eta_i=2\hat\omega_i+\hat S\hat\eta_j\wedge\hat\eta_k,
$
(cf. formula~\ref{streq}), and the vector fields $J_sN$ differ from the Reeb vector fields $\hat\xi_s.$

\section{QC-hypersurfaces of hyper-K\"ahler manifolds}\label{S:qc in hyperkahler}

Let $M$ be a qc-hypersurface of the hyper-K\"ahler manifold $K$ as in subsection \ref{ss:qc-hypersurfaces}. Summarizing the notation of \ref{ss:qc-hypersurfaces} we have that the defining tensors of the embedded qc-structure on $M$ are assumed to be given by
\begin{equation}\label{e:def of hat str}
\hat\eta_s (A)= G(J_sN,A), \quad \hat \xi_s=J_s N +\hat r_s, \quad \hat\omega_s(X,Y)=-II(I_sX,Y), \quad \hat g(X,Y)= \hat\omega_s(I_sX,Y).%, \ s=1,2,3.
\end{equation}

Notice that Theorem \ref{t:qc-embedded in hyperkahler} claims that the qc-conformal class of any embedded  qc-hypersurface in a hyper-K\"ahler manifold contains a qc-Einstein structure. In turn, this follows from the following stronger result.
\begin{thrm}\label{t:surface1}
Let $\iota:M\rightarrow K$ be an oriented qc-hypersurface of a hyper-K\"ahler
manifold $K$ with parallel quaternion structures $J_s$, $s\in\{1,2,3\}$, and hyper-K\"ahler metric $G$. There exists a unique up to a multiplicative constant symmetric $J_s$-invariant bilinear form $\Delta$ on the pull-back bundle $TK|_M\overset{def}{=}\iota^*(TK)\rightarrow
M$ such that $\Delta$ is parallel with respect to the pull-back of the Levi-Civita connection and whose restriction to $TM$ is proportional to the second fundamental form of $M$. Furthermore, the restriction of $\Delta$ to $H$ is the horizontal metric of a qc-Einstein structure in the qc-conformal class defined by the (second fundamental form of the) qc-embedding.
\end{thrm}

We note that the existence is the main difficulty in the above result, since the uniqueness up to a multiplicative constant is trivial. Indeed, if $\Delta_1$ and $\Delta_2$ are two such forms, then from $\Delta_1{|_{TM}}=e^{2\phi}\Delta_2{|_{TM}}$ for some function $\phi$ on $M$, the $J_s$-invariance implies the same relation  on $TK|_M$. Therefore, $d\phi(A)=0$ for any $A\in TM$ since the bilinear forms are parallel.

Before we turn to the proof of Theorem \ref{t:surface1}, we give an example of the above construction and Theorem \ref{t:surface1}  by considering the standard embedding of the Heisenberg group in the $n+1$-dimensional quaternion space.
\begin{exm}\label{e:Heisenberg embed} An embedding of the quaternionic Heisenberg group $\QH$, see \cite[Section 5.2]{IMV1}.
\end{exm}
\noindent Let us identify $\QH$ with the boundary $\Sigma$ of a Siegel domain in
$\Hn\times\mathbb{H}$, $
\Sigma\ =\ \{ (q',p')\in \Hn\times\mathbb{H}\ :\ \Re {\ p'}\ =\ -\,\abs{q'}^2 \},
$
by using the map $\iota\left ((q', \omega')\right ) \ =\  (q',-\,\abs{q'}^2 +\omega')=(q,p)\in  \Hn\times\mathbb{H},$ where $p=t+\omega=t+ix+jy+kz\in \mathbb{H}$, $q=(q_1,\dots, q_n)\in \Hn$, and $q_\alpha=t_\alpha +ix_\alpha +jy_\alpha +kz_\alpha \in \mathbb{H}$, $\alpha=1,\dots,n$. The "standard" contact
form on $\QH$, written as a purely imaginary quaternion valued form,  is  given by
\begin{equation}\label{e:stand contact form on H}
\begin{aligned}
\tilde{\Theta} & = \frac 12\ (-d\omega   +  d\bar q\, \cdot q -  q \cdot d\bar q)=
 i\left (- \frac 12 dx -  t_\alpha dx_\alpha +  x_\alpha dt_\alpha + y_\alpha dz_\alpha - z_\alpha dy_\alpha \right)\\
& +j\left (-\frac 12 dy -  t_\alpha dy_\alpha -  x_\alpha dz_\alpha + y_\alpha dt_\alpha + z_\alpha dx_\alpha \right) +  k\left (-\frac 12 dz -  t_\alpha dz_\alpha +  x_\alpha dy_\alpha - y_\alpha dx_\alpha + z_\alpha dt_\alpha\right),
\end{aligned}
\end{equation}
where $\cdot$ denotes the quaternion multiplication. We note that
the complex structures $J_1, \, J_2, \, J_3$ are, respectively, the  multiplication on the \emph{right} by $-i, \, -j,\, -k $ in $\mathbb{H}^{n+1}$, hence
\begin{equation*}%\label{e:cmplx str}
\begin{aligned}
& J_1 dt_\alpha =-dx_\alpha, \quad J_1 dy_\alpha = dz_\alpha, \quad J_1 dt =-dx,\quad J_1 dy = dz,\\
& J_2dt_\alpha = -dy_\alpha, \quad J_2dz_\alpha= dx_\alpha,\quad J_2dt = -dy, \quad J_2dz= dx.\
%& J_3dt_\alpha = -dz_\alpha, \quad J_3dx_\alpha =dy_\alpha, \quad J_3dt = -dz, \quad J_3dx =dy.
\end{aligned}
\end{equation*}
Clearly, $\Sigma$ is the $0$-level set of $\rho= \abs{q}^2 +t$ and we have
\begin{equation*}%\label{e:Heisenberg forms}
\begin{aligned}
& J_sd\rho= \sqrt{1+4|q|^2}\, \hat\eta_s,\quad
 N=\frac {2}{\sqrt{1+4|q|^2}}\left (\frac 12\dt +\  t_\alpha\dta + \ x_\alpha\dxa +\ y_\alpha\dya +\ z_\alpha\dza \right ),\\
 \hat\eta & = i\hat\eta_1+j\hat\eta_2+ k\hat\eta_3 =\frac {1}{\sqrt{1+4|q|^2}}\ (-d\omega \ + \ d\bar q\, \cdot q \ - \ \bar q \cdot d q ), \\
& II(A,B)=-\frac {1}{|d\rho|}\, Dd \rho \, (A,B)= -\frac {2}{\sqrt{1+4|q|^2}}\,\langle A_H, B_H \rangle\\
&\hskip.6truein=-\frac {2}{\sqrt{1+4|q|^2}}\,\left ( dt_\alpha\odot dt_\alpha +dx_\alpha\odot dx_\alpha+dy_\alpha\odot dy_\alpha+dz_\alpha\odot dz_\alpha \right) \, (A,B),
\end{aligned}\end{equation*}
where for a tangent vector $A$ we use $ A_H=A-dt(A)\dt - dx(A)\dx  -dy(A)\dy - dz(A)\dz$ for the orthogonal projection from $\mathbb{H}^{n+1}$ to
the horizontal space, which is given by $H= Ker\, {d\rho}\cap\,\{ \cap_{s=1}^3 Ker\,\hat\eta_s\}$.  From the above formulas we see that $\Theta\overset{def}{=}\iota^*\hat\eta$ is conformal to $\tilde\Theta$. Therefore, the qc-structure $\eta_s=\frac {\sqrt{1+4|q|^2}}{2}\hat\eta_s $, i.e., the standard qc-structure \eqref{e:stand contact form on H}, has  horizontal metric given by the restriction of the bilinear form $\Delta=const\,\Re(dq_{\alpha}\cdot d\bar q_{\alpha})\vert_{M}$, which is parallel along $M$. This is the symmetric form whose existence is claimed by Theorem \ref{t:surface1}, while the calibrating function is a certain multiple of $\sqrt{1+4|q|^2}$, cf. \eqref{def:f1}.

It is worth noting that the qc-Einstein structures in the qc-conformal class of the standard qc-structure were essentially classified in \cite[Theorem 1.1]{IMV1} where it shown that all qc-Einstein structures globally conformal to the standard qc-structure  are obtained from the standard with a qc-automorphism.

\subsection{Proof of Theorem~\ref{t:surface1}.} A key point of our analysis is a  volume-normalization condition, which is based on Lemma \ref{mu}.
To this effect we consider  a qc-conformal transformation $\eta_s=f\hat\eta_s$ where $f$ is a smooth function on
$M$. Let $\xi_s$, $\omega_s$, $\nabla$ and $\alpha_s$
be the Reeb vector fields, the fundamental 2-forms, the Biquard
connection and the $\spl$-connection 1-forms of the deformed
admissible set. The orthogonal complement
$V=\text{span}\{\xi_1,\xi_2,\xi_3\}$ of $H$ and the endomorphism
$I_1$, defined on the horizontal space $H$,  induce a
decomposition of the complexified tangent bundle of $M$ (we use
the same notation $TM$ for both the tangent bundle and its
complexification), $TM=V\oplus H^{1,0}_{I_1}\oplus H^{0,1}_{I_1},$
and consequently of the whole complexified tensor bundle
of $M$. We shall need the type decomposition of the one and two-forms on $M$,
\begin{align*}
T^*M\ &=\ H^*_{1,0}\oplus H^*_{0,1}\oplus L^*, \qquad L^*=\text{span}\{\eta_1,\eta_2,\eta_3 \},\\
\Lambda^2(T^*M)\ &=\
\Lambda^2(H^*_{1,0})\oplus\Lambda^2(H^*_{0,1})\oplus
(H^*_{1,0}\otimes H^*_{0,1}) \oplus \Lambda^2(L^*)\oplus
(L^*\otimes H^*).
\end{align*}
In particular, $H^{*}_{1,0}$ is the $2n$-dimensional space of all complex one-forms  which
 vanish on $\xi_1,\xi_2,\xi_3$ and are of
 type $(1,0)$ with respect to $I_1$ when restricted to $H$.
Similarly, using the endomorphism $I_2$ or
$I_3$ we obtain corresponding decompositions. We shall write explicitly the analysis with respect to $I_1$, but keep in mind that the
arguments remains true if we cyclicly permute the indices $1,2$ and $3$.

Consider the following complex 2-forms on $M$,
\begin{equation*}
\begin{aligned}
\hat\gamma_i=\hat\omega_j+\sqrt{-1}\ \hat\omega_k,\qquad
\gamma_i=\omega_j+\sqrt{-1}\ \omega_k,\qquad
{\Gamma}_i(\A,\B)\ = \ G(J_j\A,\B)+\sqrt{-1}\, G(J_k\A,\B).
\end{aligned}
\end{equation*}
We have $\gamma_s=f\hat\gamma_s$,
$\xi_t\lrcorner\gamma_s=0$ and $\gamma_1,
\hat\gamma_1|_H, \Gamma_1|_H\in\Lambda^2(H^*_{1,0})$ .
Moreover, since $N$ is a hyper-K\"ahler
manifold, the three 2-forms $\Gamma_s$ are closed, $d\Gamma_s=0$.
The volume normalization relies on the following algebraic
lemma.

\begin{lemma}\label{mu}
Let $\mathcal H^{4n}$ be a real vector space with
hyper-complex structure $(I_1,I_2,I_3)$, i.e.,
$I_1^2=I_2^2=I_3^2=-Id,\ I_1I_2=-I_2I_1=I_3$ and $\hat g$ and $g$ be two positive definite inner products on
$\mathcal H^{4n}$ satisfying $\hat g(I_sX,I_sY)=\hat g(X,Y)$, and
$g(I_sX,I_sY)=g(X,Y)$ for all  $X,Y\in \mathcal
H^{4n}$, $s=1,2,3$. If
\begin{equation*}
\hat\gamma_i(X,Y)= \hat g(I_jX,Y)+ \sqrt{-1}\, \hat g(I_kX,Y),\qquad
\gamma_i(X,Y)= g(I_jX,Y)+ \sqrt{-1}\, g(I_kX,Y),
\end{equation*}
then  there exists a positive real number $\mu$ such that
$\underset{n~\text{times}}{\underbrace{\hat\gamma_s\wedge\dots\wedge \hat\gamma_s}}=\mu\,\underset{n~\text{times}}{\underbrace{(\gamma_s\wedge\dots\wedge \gamma_s)}}$, $s=1,2,3.$
\end{lemma}

\begin{proof}%[\bf Proof of Lemma~\ref{mu}]
%We note that the real number $\mu$ is the same for all $s=1,2,3.$
A small calculation shows that both $\gamma_1$ and $\hat\gamma_1$
are of type $(2,0)$ with respect to $I_1$. The complex vector
space $\Lambda^{2n}(\mathcal H^*_{1,0})$ is one dimensional, and
$\gamma_1^{n}$ and $ \hat\gamma_1^{n}$ are non zero elements of
it, hence there exists a non zero complex number $\mu$ such that
$\gamma_1^{n}=\mu\, \hat\gamma_1^{n}.$ Note that
$I_2\gamma_1=\overline{\gamma_1}$ and the same holds true for
$\hat\gamma_1$. It follows that
\begin{equation*}
(I_2\gamma_1)^n= \overline {\gamma_1^n}\,   \ \text{ i.e., } \
\mu\hat\gamma_1^n = \bar \mu\,  \overline {\hat\gamma_1^n},
\end{equation*}
thus $\mu=\bar\mu\ne 0.$  The group $GL(n,\mathbb H)$ acts
transitively on the set of all positive definite inner products
$g$ of $\mathcal H$, compatible with the hyper-complex structure,
and hence also on the set of all corresponding 2-forms $\gamma_1$.
The group $GL(n,\mathbb H)$ is connected, therefore each orbit
is connected as well, which implies $\mu>0.$ It remains
to show that the constant $\mu$ in the equation
$\hat\gamma_s^{n}=\mu\,\gamma_s^{n}$ is independent of $s$. For
this we use that the $4n$-form $\gamma_1^{n}\wedge
\overline{\gamma_1^{n}}$ equals the volume form $vol(g)$ of the
metric $g$ and hence it is independent of $s$. This implies that
$\mu^2$ does not depend on $s$ and therefore the same is true for
$\mu$.
\end{proof}
From Lemma~\ref{mu} applied to the metrics $\hat g$ and
$G|_H$ on $H$ it follows that there exists a positive function
$\mu$ on $M$ such that $\Gamma_s^n|_H=\mu\hat\gamma^n_s|_H,$ $s=1,2,3$ i.e.,
\begin{equation}\label{e:def pf mu}
{\Gamma}_s^n \equiv \mu \hat\gamma_s^n \mod \{\eta_1,\eta_2,\eta_3\}.
%, \qquad s=1,2,3.
 \end{equation}
At this point we define the "\emph{calibrated}" qc-structure  using the function $f$ defined by
\begin{equation}\label{def:f1}
f=\mu^{\frac{1}{n+2}}.
 \end{equation}
The reminder of this section is devoted to showing that with this choice of $f$ the qc-structure determined by $\eta_s$ satisfies all the requirements of the theorem.

We start by proving in Lemma~\ref{tech-lemma} a few important preliminary technical facts. Let us define the following three vector fields $r_s$,
\begin{equation}\label{e:def of the r's}
r_s=\xi_s-\frac{1}{f}\,J_sN,\ s=1,2,3.
\end{equation}
Since $\eta_t(r_s)=\delta_{ts}-\hat \eta_t(J_sN)=0$, $s,t=1,2,3,$
it follows that $r_s$ are  horizontal vector field, $r_s\in H$. We
will denote by $r_s$ also the corresponding 1-forms, defined by
$r_s(\A)=G(r_s,\A),\ \A\in TM$.
\begin{rmrk}
Note that in general expressions of the type
$\eta_1\wedge\eta_2\wedge\eta_3\wedge\delta$, with $\delta$ being
differential form on $M$, depend only on the restriction of
$\delta$ to $H$. This fact will be used repeatedly hereafter.
\end{rmrk}

\begin{lemma}\label{tech-lemma} We have
\begin{align}\label{basic eq}
\eta_2\wedge{\Gamma}_1^{n+1}\ &=\
(n+1)\,\eta_1\wedge\eta_2\wedge\eta_3\wedge\gamma_1^n\
,\\\label{basic eq2} \Gamma_1^{n+1}\ &=\
\sqrt{-1}(n+1)\eta_1\wedge(\eta_2+\sqrt{-1}\eta_3)\wedge\gamma_1^n\
\\\nonumber &\quad+\
n(n+1)f^{-2}\eta_1\wedge\eta_2\wedge\eta_3\wedge(-J_3r_3+
\sqrt{-1}J_2r_3+J_2r_2+\sqrt{-1}J_3r_3)\wedge\Gamma_1^{n-1}\ .
\end{align}
Furthermore, the above equations hold after any cyclic permutation of the indices $1,2$ and $3$.
\end{lemma}

\begin{proof}%[Proof of Lemma \ref{tech-lemma}]
Let us define ${\Gamma}^{'}_1$ and ${\Gamma}^{''}_1$ to be 2-forms
on $M$ which coincide with the 2-form $\Gamma_1$ when restricted
to the distribution $H$ and satisfy the additional  conditions
$\xi_s\lrcorner{\Gamma}^{'}_1=0$,
$(J_sN)\lrcorner{\Gamma}^{''}_1=0$. In order to
find the relation between $\Gamma_1$ and $\Gamma_1^{'}$, we
compute
\begin{multline*}
\Gamma'_1 ({\A},{\B})= \Gamma_1 ({\A} -\eta_s({\A})
\xi_s,{\B}-\eta_t({\B})\xi_t)=\Gamma_1({\A},{\B}) - \eta_s({\B})
\Gamma_1({\A},\xi_s)-\eta_s ({\A})\Gamma_1
(\xi_s,{\B})+
\Gamma_1(\xi_s,\xi_t)\eta_s({\A})\eta_t({\B})\\
= \Gamma_1({\A},{\B})-
\eta_t\wedge(\xi_t\lrcorner\Gamma_1)({\A},{\B})
+\frac 12 \Gamma_1(\xi_s,\xi_t)\eta_s\wedge\eta_t ({\A},{\B}).
\end{multline*}
A short calculation gives
\begin{multline*}
(\xi_t\lrcorner\Gamma_1)(\A)=G(J_2\xi_t,\A)+\sqrt{-1}G(J_3\xi_t,\A)=G(J_2(r_t+\frac{1}{f}J_tN),\A)+\sqrt{-1}+G(J_3(r_t+\frac{1}{f}J_tN),\A)
\\
=(J_2r_t+J_3r_t)(\A) \mod \{\eta_1,\eta_2,\eta_3\},
\end{multline*}
which shows that for  some functions $\Gamma^{s,t}_1$ on $M$ we have
\begin{equation}\label{gamma'}
{\Gamma^{'}}_1={\Gamma}_1-
\sum_{t=1}^3\eta_t\wedge(J_2r_t+\sqrt{-1}\,J_3r_t)+ \sum_{s,\,
t=1}^3{\Gamma}^{s,t}_1\eta_s\wedge\eta_t.
\end{equation}
Similarly to the derivation of \eqref{gamma'} we can find the relation between ${\Gamma}_1^{''}$ and ${\Gamma}_1$,
\begin{equation*}%\label{e:comformal Gamma}
{\Gamma}_1^{''}={\Gamma}_1 \ - \
f^{-2}(\eta_3\wedge\eta_1+\sqrt{-1}\,\eta_1\wedge\eta_2),
\end{equation*}
which gives
\begin{equation}\label{gamma_n+1}
\Gamma_1^{n+1}=
\sqrt{-1}\,{(n+1)}{f^{-2}}\,\eta_1\wedge(\eta_2+\sqrt{-1}\,\eta_3)
\wedge({\Gamma}_1^{''})^n =
\sqrt{-1}\,{(n+1)}{f^{-2}}\,\eta_1\wedge(\eta_2+\sqrt{-1}\,\eta_3)\wedge{\Gamma}_1^n.
\end{equation}
Clearly, $\Gamma_s^{'}\in \Lambda^2(H^*_{1,0})$ and
$(\Gamma_s^{'})^{n+1}=(\Gamma_s^{''})^{n+1}=0$. Noting that
\eqref{e:def pf mu} are equivalent to the equations
\begin{equation*}%\label{def:f}
(\Gamma_s^{'})^n=f^2\gamma_s^n
\mathfrak{}\end{equation*}
we obtain from \eqref{gamma'} the identity
\begin{align}\label{gamma'n}
{\Gamma}_1^n\ &=\ ({\Gamma}_1^{'})^n\ +\ n\sum_{s=1}^3\eta_s
\wedge(J_2r_s+\sqrt{-1}\,J_3r_s)\wedge({\Gamma}_1^{'})^{n-1}\qquad mod\ < \eta_s\wedge\eta_t>\\\nonumber &= \
f^2\gamma_1^n\ +\ n\sum_{s=1}^3\eta_s
\wedge(J_2r_s+\sqrt{-1}\,J_3r_s)\wedge({\Gamma}_1^{'})^{n-1}\qquad mod\ < \eta_s\wedge\eta_t>.
\end{align}
Finally, a substitution of \eqref{gamma'n} in
\eqref{gamma_n+1} gives
\begin{align*}
\Gamma_1^{n+1}\ &=\
\sqrt{-1}(n+1)\eta_1\wedge(\eta_2+\sqrt{-1}\eta_3)\wedge\gamma_1^n\
\\\nonumber &\quad+\
n(n+1)f^{-2}\eta_1\wedge\eta_2\wedge\eta_3\wedge(-J_3r_3+
\sqrt{-1}J_2r_3+J_2r_2+\sqrt{-1}J_3r_3)\wedge(\Gamma^{'}_1)^{n-1},
\end{align*}
which, in view of the relation
$\eta_1\wedge\eta_2\wedge\eta_3\wedge(\Gamma^{'}_1)^{n-1}=
\eta_1\wedge\eta_2\wedge\eta_3\wedge\Gamma_1^{n-1}$, yields
\eqref{basic eq2}. The equation \eqref{basic eq} follows now by
taking the wedge products of  both sides of \eqref{basic eq2} with
the 1-form $\eta_2$.

\end{proof}

Following is a technical lemma which will be used in the proof of Lemma \ref{l:step1} below.% in order to simplify \eqref{split_basic1.1}
\begin{lemma}\label{omega-gamma}
For any $\lambda\in H^*_{1,0}$ (considered with respect to $I_1$)
we have
$$\lambda\wedge\omega_1\wedge\gamma_1^{n-1}=\frac{\sqrt{-1}}{2n}(I_2\lambda)\wedge\gamma_1^n.$$
\end{lemma}
\begin{proof}%[\bf Proof of Lemma~\ref{omega-gamma}]
We can take a basis of the cotangent space of $M$ in the form
$$\eta_1,\eta_2,\eta_3,\epsilon_1,...,\epsilon_n,I_1\epsilon_1,...,I_1\epsilon_n,
I_2\epsilon_1,...,I_2\epsilon_n,I_3\epsilon_1,\dots,I_3\epsilon_n,$$
where $\xi_s\lrcorner\epsilon_t=0,s=1,2,3$, $t=1,2,\dots,n,$ which is orthonormal in the sense  that the following equations hold
\[
\omega_1 = \sum_{s=1}^n(\epsilon_s\wedge
I_1\epsilon_s+I_2\epsilon_s\wedge
I_3\epsilon_s),\quad
\omega_2 = \sum_{s=1}^n(\epsilon_s\wedge
I_2\epsilon_s+I_3\epsilon_s\wedge
I_1\epsilon_s),\quad
\omega_3 = \sum_{s=1}^n(\epsilon_s\wedge
I_3\epsilon_s+I_1\epsilon_s\wedge I_2\epsilon_s).
\]
For $\phi_t=\epsilon_t+\sqrt{-1}I_1\epsilon_t$ and
$\psi_t=I_2\epsilon_t+\sqrt{-1}I_3\epsilon_t$ the forms
$\phi_1,\dots,\phi_n,\psi_1,...,\psi_n$ form a basis of $H^*_{1,0}.$
Moreover, we have
\begin{align*}
&I_2\phi_s=\bar\psi_s,\ \ I_2\psi_s=-\bar\phi_s,\  s=1,\dots,n, \quad\omega_1 = \frac{\sqrt{-1}}{2}\sum_{s=1}^n(\phi_s\wedge\bar\phi_s+\psi_s\wedge\bar\psi_s),\quad
 \gamma_1 =  \sum_{s=1}^n\phi_s\wedge\psi_s,\\
  &\gamma_1^n =n!\,\phi_1\wedge\psi_1\wedge\dots\wedge\phi_n\wedge\psi_n, \quad\gamma_1^{n-1} =
(n-1)!\sum_{s=1}^n\phi_1\wedge\psi_1\wedge\dots\widehat{\wedge\phi_s
\wedge\psi_s\wedge}\dots\wedge\phi_n\wedge\psi_n,\\
&\omega_1\wedge\gamma_1^{n-1} =
\frac{\sqrt{-1}(n-1)!}{2}\sum_{s=1}^n\phi_1\wedge\psi_1\wedge\dots\wedge(\phi_s\wedge\bar\phi_s+\psi_s\wedge\bar\psi_s)\wedge...
\wedge\phi_n\wedge\psi_n.
\end{align*}
Since $\lambda\in H^*_{1,0}$ there exist constants $a_s$, $b_s$, $s=1,\dots,n$
such that $\lambda=\sum_{s=1}^n(a_s\phi_s+b_s\psi_s)$. It follows that
$I_2\lambda=\sum_{s=1}^n(a_s\bar\psi_s-b_s\bar\phi_s).$ Finally we
compute (omitting the sum symbols)
\begin{gather*}
\lambda\wedge\omega_1\wedge\gamma_1^n
=\frac{\sqrt{-1}(n-1)!}{2}(a_t\phi_t+b_t\psi_t)
\wedge(\phi_1\wedge\psi_1\wedge\dots\wedge(\phi_s\wedge\bar\phi_s+\psi_s\wedge\bar\psi_s)\wedge\dots
\wedge\phi_n\wedge\psi_n) \\
=\frac{\sqrt{-1}(n-1)!}{2}(a_s\bar\psi_s-b_s\bar\phi_s)\wedge\phi_1\wedge\psi_1\wedge\dots
\wedge\phi_n\wedge\psi_n=\frac{\sqrt{-1}}{2n}(I_2\lambda)\wedge\gamma_1^n.
\end{gather*}
\end{proof}

\begin{lemma}\label{l:step1} The calibrated qc-structure
$\eta_s=f\hat\eta_s$, where $f$ is given by
\eqref{def:f1}, satisfies the structure equations
\eqref{str_eq_mod}. In particular, $(M,H,\eta_s)$ is a qc-Einstein structure. Furthermore, we have $$I_1r_1=I_2r_2=I_3r_3.$$
\end{lemma}

\begin{proof}
Taking the exterior derivative of $\eqref{basic eq}$ and recalling that $\Gamma_1$ is a closed form, we obtain
\begin{multline}\label{d(besic eq)}
 =\ n(n+1)\eta_1\wedge\eta_2\wedge\eta_3\wedge
d\gamma_1\wedge\gamma_1^{n-1}+d\eta_2\wedge({\Gamma}_1^{n+1}+(n+1)\eta_1\wedge(\eta_3-\sqrt{-1}\eta_2)\wedge\gamma_1^n)\ \\
-\ (n+1)\left ( d\eta_1\wedge\eta_2\wedge(\eta_3-\sqrt{-1}\eta_2)\
+\
\eta_1\wedge\eta_2\wedge d(\eta_3-\sqrt{-1}\eta_2)\right )\wedge\gamma_1^n .
\end{multline}
The structure equations \eqref{streq} and the identities $\omega_2= \frac{1}{2}(\gamma_1+\bar\gamma_1)$, $
\omega_3=\frac{\sqrt{-1}}{2}(\bar\gamma_1-\gamma_1)$ and
$\omega_1\wedge\gamma_1^n=0$
imply
\begin{align*}
&d\eta_1\equiv 0 \mod \{\eta_2, \eta_3, H^{*}_{1,0}\},\quad d\eta_2\equiv \bar\gamma_1 \mod \{\eta_1,\eta_3, H^{*}_{1,0}\},\quad
d\eta_3\equiv \sqrt{-1}\bar\gamma_1 \mod \{\eta_1,\eta_2, H^{*}_{1,0}\},\\
&d(\eta_3-\sqrt{-1}\eta_2)\equiv-2\sqrt{-1}\gamma_1+\sqrt{-1}\eta_3\wedge\alpha_1
\mod \{\eta_1,\eta_2\},\\
&d\gamma_1\equiv-\sqrt{-1}\alpha_1\wedge\gamma_1+(-\alpha_3+\sqrt{-1}\alpha_2)\wedge\omega_1
\mod \{\eta_1,\eta_2,\eta_3\}.
\end{align*}
From \eqref{basic eq2} and the above identities applied to
\eqref{d(besic eq)} we find
\begin{multline*}
0\ =\
\sqrt{-1}n(n+1)d\eta_2\wedge\eta_1\wedge(\eta_2+\sqrt{-1}\eta_3)
\wedge f^{-2}\sum_{s=1}^3\eta_s\wedge(J_2r_s+\sqrt{-1}J_3r_s)\wedge\Gamma_1^{n-1}\ \\
-\
(n+1)\eta_1\wedge\eta_2\wedge\sqrt{-1}\eta_3\wedge\alpha_1\wedge\gamma_1^n\
-\
n(n+1)\eta_1\wedge\eta_2\wedge\eta_3\wedge\sqrt{-1}\alpha_1\gamma_1^n\
\\+\
n(n+1)\eta_1\wedge\eta_2\wedge\eta_3\wedge(\alpha_3+\sqrt{-1}\alpha_2)\wedge\omega_1
\wedge\gamma_1^{n-1}\
\end{multline*}
\begin{multline*}
=\ n(n+1)f^{-2}\eta_1\wedge\eta_2\wedge\eta_3\wedge
\bar\gamma_1\wedge{\Gamma}_1^{n-1}
\wedge(-J_3r_3+\sqrt{-1}J_2r_3+J_2r_2+\sqrt{-1}J_3r_2)\ \\
 +\
n(n+1)\eta_1\wedge\eta_2\wedge\eta_3\wedge(-\alpha_3+\sqrt{-1}\alpha_2)\wedge\omega_1\wedge\gamma_1^{n-1}\
\ -
(n+1)^2\eta_1\wedge\eta_2\wedge\eta_3\wedge\gamma_1^n\wedge\alpha_1.
\end{multline*}
The last expression is a $(2n+4)$-form which belongs to  the space (decomposition with
respect to $I_1$)
$$\Lambda^3(L^*)\otimes\Lambda^{2}(H^*_{0,1})\otimes\Lambda^{2n-1}(H^*_{1,0})\ \oplus\
\Lambda^3(L^*)\otimes\Lambda^{1}(H^*_{0,1})\otimes\Lambda^{2n}(H^*_{1,0}).$$
Hence, we obtain the next two identities
\begin{multline}\label{split_basic1}
(n+1)^2\eta_1\wedge\eta_2\wedge\eta_3\wedge\gamma_1^n\wedge\alpha_1\
\\
=\
n(n+1)\eta_1\wedge\eta_2\wedge\eta_3\wedge\frac{1}{2}(-\alpha_3-\sqrt{-1}I_1\alpha_3+\sqrt{-1}\alpha_2-I_1\alpha_2)
\wedge\omega_1\wedge\gamma_1^{n-1}.
\end{multline}
and also
\begin{multline}\label{split_basic2}
-\
n(n+1)f^{-2}\eta_1\wedge\eta_2\wedge\eta_3\wedge\bar\gamma_1\wedge{\Gamma}_1^{n-1}
\wedge(-J_3r_3+\sqrt{-1}J_2r_3+J_2r_2+\sqrt{-1}J_3r_2)\ \\
=\
n(n+1)\eta_1\wedge\eta_2\wedge\eta_3\wedge\frac{1}{2}(-\alpha_3+\sqrt{-1}I_1\alpha_3+\sqrt{-1}\alpha_2+I_1\alpha_2)
\wedge\omega_1\wedge\gamma_1^{n-1}.
\end{multline}
Equation \eqref{split_basic1} yields
\begin{equation}\label{split_basic1.1}
n(-\alpha_3-\sqrt{-1}I_1\alpha_3+\sqrt{-1}\alpha_2-I_1\alpha_2)\wedge\omega_1\wedge\gamma_1^{n-1}\equiv (n+1)\gamma_1^n\wedge(\alpha_1-\sqrt{-1}I_1\alpha_1)
\mod\{\eta_1,\eta_2,\eta_3\}.
\end{equation}

With the help of Lemma \ref{omega-gamma} we can write \eqref{split_basic1.1} in the form
\[
\frac{\sqrt{-1}}{2}I_2(-\alpha_3-\sqrt{-1}I_1\alpha_3
+\sqrt{-1}\alpha_2-I_1\alpha_2)
\equiv (n+1)(\alpha_1-\sqrt{-1}I_1\alpha_1)\quad
\mod\quad \{\eta_1,\eta_2,\eta_3\}.
\]
Taking the real part of the last identity we come to
%\begin{equation*}
$2(n+1)I_1\alpha_1+I_2\alpha_2+I_3\alpha_3\equiv 0\quad \mod\quad
\{\eta_1,\eta_2,\eta_3\}.$
%\end{equation*}
A cyclic rotation of the indices $1,2,3$ in the above arguments gives the following system  $mod\
\{\eta_1,\eta_2,\eta_3\}$
\begin{equation*}
\begin{aligned}
&2(n+1)I_1\alpha_1+I_2\alpha_2+I_3\alpha_3\equiv0 \\
&I_1\alpha_1+2(n+1)I_2\alpha_2+I_3\alpha_3\equiv0   \\
&I_1\alpha_1+I_2\alpha_2+2(n+1)I_3\alpha_3\equiv0,
\end{aligned}
\end{equation*}
which has the unique solution $I_1\alpha_1\equiv I_2\alpha_2\equiv
I_3\alpha_3\equiv 0 \mod\ \{\eta_1,\eta_2,\eta_3\}$. Therefore, the calibrated qc-structure has vanishing $\spl$-connection 1-forms
\begin{equation}\label{e:calibrated vanish alpha}
(\alpha_1)|_H=(\alpha_2)|_H=(\alpha_3)|_H=0,
\end{equation}
hence
by \eqref{e:vanishing alphas qc-Einstein} it is a qc-Einstein structure. From  \eqref{split_basic2} (and a cyclic rotation of the indeces) we also conclude  that $I_1r_1=I_2r_2=I_3r_3$.
\end{proof}

We shall denote by $r$ the common vector defined above by  $I_sr_s$ in Lemma \ref{l:step1}, see also \eqref{e:def of the r's},
\begin{equation*}%\label{e:r def}
r=-I_sr_s\in H, \qquad \text{hence}\ r_s=I_sr.
\end{equation*}
The calibrated qc-structure constructed in Lemma \ref{l:step1} enjoys further useful technical properties recorded below.

\begin{lemma}\label{l:qc-einstein embedded}
The second fundamental form $II$ of the qc-embedding $M\subset K$ and the calibrating function $f$ defined by \eqref{def:f1} satisfy the identities:
\par i) $II(X,Y)=-f^{-1}g(X,Y)$;
\par ii) $II(J_sN,J_sX)=-f^{-1}df(X)=g(r,X), \ X\in H$;
\par iii) $II(J_sN,J_tN)=-\delta_{st}f(S/2+g(r,r))$;
\par iv)  $df(J_sN)=df(\xi_s)=0$.
\end{lemma}

\begin{proof}
 {\it (i).} The identity $II(X,Y)=-f^{-1}g(X,Y)$
holds by the definition of~$g$, also recall \eqref{e:def of hat str}.

 {\it (ii).} Using the fact that the
complex structures $J_s$ are $D$-parallel, the relation
$\eta_s=fG(J_sN,.)$  and the formula $d\eta_s(\A,\B)=(D_{\A}\eta_s)(\B)-(D_{\B}\eta_s)(\A)$ we find
\begin{equation} \label{II-deta}
d\eta_s(\A,\B)=f^{-1}df\wedge\eta_s (\A,\B) + f II(\A, [J_s\B]_{TM})-f II(\B,[J_s\A]_{TM}).
\end{equation}
%We apply the above formula with $(s,\A,\B)=(i,J_jN,J_kX)$ and $(i,J_iN,X)$ to get
The above formula implies
\begin{equation}\label{II-deta-1}
%\begin{aligned}
d\eta_i(J_jN,J_kX)= -fII(J_jN,J_jX)-fII(J_kN,J_kX),\qquad d\eta_i(J_iN,X)= -df(X)+fII(J_iN,J_iX).
%\end{aligned}
\end{equation}
On the other hand, since $\xi_s=\frac{1}{f}J_sN+J_sr$ and $\alpha_i{_\vert{_H}}=(\xi_j\lrcorner d\eta_k){_\vert{_H}}=0$,
we have
\begin{equation}\label{II-deta-2}
%\begin{aligned}
0=d\eta_i(\xi_j,J_kX)=f^{-1}d\eta_i(J_jN,J_kX)+2g(r,X),\qquad 0=d\eta_i(\xi_i,X)=f^{-1}d\eta_i(J_iN,X)-2g(r,X).
%\end{aligned}
\end{equation}
The first of the above identities together with the first identity in
\eqref{II-deta-1} imply the equation $II(J_iN,J_iX)=g(r,X)$,
which together with the second identity in \eqref{II-deta-1} and
\eqref{II-deta-2} give the identities in (ii).

{\it (iii) and (iv).} From \eqref{II-deta}
%with $(s,\A,\B)=(i,J_iN,J_jN),\ (i,J_iN,J_kN)$ and $(i,J_jN,J_kN)$ to
we have
\begin{gather}\label{II-deta-3}
%\begin{aligned}
d\eta_i(J_iN,J_jN)=-df(J_jN)+fII(J_iN,J_kN), \quad
 d\eta_i(J_iN,J_kN)=-df(J_kN)-fII(J_iN,J_jN),\\
 d\eta_i(J_jN,J_kN)=-fII(J_jN,J_jN)-fII(J_kN,J_kN),\notag
%\end{aligned}
\end{gather}
which give the wanted identities.
From \eqref{e:calibrated vanish alpha}, \eqref{streq}, \eqref{str_eq_mod} and \eqref{e:vanishing alphas qc-Einstein} we have $d\eta_s(\xi_j,\xi_k)=2\delta_{si}S$. Therefore, we obtain
\begin{equation}\label{II-deta-4}
\begin{aligned}
&
0=d\eta_i(\xi_i,\xi_j)=d\eta_i(f^{-1}J_iN+J_ir,f^{-1}J_jN+J_jr)=f^{-1}d\eta_i(J_iN,J_jN)\\
&
0=d\eta_i(\xi_i,\xi_k)=d\eta_i(f^{-1}J_iN+J_ir,f^{-1}J_kN+J_kr)=f^{-1}d\eta_i(J_iN,J_kN)\\
& S=d\eta_i(\xi_j,\xi_k)=f^{-2}d\eta_i(J_jN,J_kN)+2g(r,r).
\end{aligned}
\end{equation}
The first two identities of \eqref{II-deta-3} and the
first two equations in \eqref{II-deta-4} give
\begin{equation*}
II(J_iN,J_jN)=-df(J_kN),\qquad II(J_jN,J_iN)=df(J_kN),
\end{equation*}
%which implies, $II(.,.)$ is symmetric tensor,
hence $df(J_kN)=0$. %$0=II(J_iN,J_jN)=df(J_kN).$
Finally, recalling \eqref{e:def of the r's}, we compute
\[
df(\xi_s)\ =\ df(r_s+f^{-1}J_sN)=
df(I_sr)=\sum_{a=1}^{4n}df(I_se_a)g(r,e_a)
=-f^{-1}\sum_{a=1}^{4n}df(I_se_a)df(e_a)= 0.
\]
The third  identity of \eqref{II-deta-3} and the
third line of \eqref{II-deta-4} imply
$$II(J_iN,J_iN)=-f(S/2+g(r,r)),$$ which completes the proof of parts (iii) and (iv) of Lemma \ref{l:qc-einstein embedded}.
\end{proof}
The next lemma gives an explicit formula for the horizontal metric of the calibrated qc-Einstein structure.
\begin{lemma}\label{l:calibrated metric}
The horizontal metric $g$ of the calibrated by \eqref{def:f1} qc-structure is related to the second fundamental form of the qc-embedding by the formula
\begin{equation}\label{II-form}
g(\A_H,\B_H)=-fII(\A,\B)-\frac {S}{2}\sum_{s=1}^3\eta_s(\A)\eta_s(\B),\quad  \A,\B\in TM,
\end{equation}
where for $\A\in TM$ we let $\A_H=A-\sum_{s=1}^3\eta_s(\A)\xi_s$ be the horizontal part of $\A$.
\end{lemma}

\begin{proof}
A few calculations give the next three identities
\begin{multline*}
II(\xi_s,X)= II(I_sr+f^{-1}J_sN,X) = II(I_sr,X) - f^{-1}II(J_sN,J_s(J_sX)) =-f^{-1}g(I_sr,X)-f^{-1}g(r,I_sX)=0,
\end{multline*}
\begin{multline*}
II(\xi_s,\xi_s)= II(I_sr+f^{-1}J_sN,I_sr + f^{-1}J_sN)= II(I_sr,I_sr)+ 2f^{-1}II(J_sN,J_sr) \\
+ f^{-2}II(J_sN,J_sN)=-f^{-1}g(r,r)+ 2f^{-1}g(r,r)-
f^{-1}(S/2+g(r,r))=-f^{-1}S/2,
\end{multline*}
\begin{multline*}
II(\xi_i,\xi_j)=II(I_ir+f^{-1}J_iN,I_jr+ f^{-1}J_jN)=II(I_ir,I_jr)+ f^{-1}II(J_iN,J_jr)\ \\
+ f^{-1}II(J_ir,J_jN)+f^{-2}II(J_iN,J_jN) =0.
\end{multline*}
The above identities together with  $II(X,Y)=-f^{-1}g(X,Y)$ yield
\eqref{II-form}, which completes the proof.
\end{proof}

At this point we are ready to complete the proof of Theorem \ref{t:surface1}. We proceed by showing that there exists a unique section  $\Delta$  of the pullback
bundle $(T^*K\otimes T^*K)|_M\rightarrow M,$ which is
$J_s$-invariant, and whose restriction to $TM$ coincides with the
tensor $-fII$.  It will be convenient to consider the \emph{calibrated transversal} to $M$ vector field
$\xi(p)=f^{-1}(p)N(p)+r(p)$, $p\in M$, which is a section of the vector bundle $TK|_M\rightarrow M$. Clearly,  $J_s\xi =\xi_s$ by \eqref{e:def of the r's}, which together with the  $J_s$ invariance of $II$ on the horizontal space $H$ gives the existence of $J_s$-invariant bilinear form on $TK|_M\rightarrow M$ by adding a bilinear form on the complement $V\oplus R\otimes \xi$. In fact, with the obvious identifications, since the fiber of $TK|_M$ over any $p\in M\subset K$ decomposes as a direct sum of subspaces as $H_p\oplus V_p\oplus R\otimes \xi(p)$, for a $v\in T_pK$ we define
$$ v'=v-\lambda(v)\xi(p)\in T_p M=H_p\oplus V_p,\qquad v''=v{'}-\sum_{s=1}^3\eta_s(v')\xi_s\in H_p,$$ where  $\lambda$ is a 1-form,  $\lambda=fG(N,.)$, so that $v'$ is the projection of $v$ on $T_p M=H_p\oplus V_p$ parallel to the calibrated transversal field $\xi$. We can rewrite formula~\eqref{II-form} in terms of the introduced decomposition as follows
\begin{equation*}
-fII(\A,\B)=g(\A'',\B'')+\frac {S}{2}\sum_{s=1}^3\eta_s(\A)\eta_s(\B),\qquad
\A,\B\in T_pM,
\end{equation*}
which leads to the following definition of the symmetric bilinear form $\Delta$,
\begin{multline}\label{delta-def}
\Delta(v,w)\overset{def}{=} -fII(v',w')+ \frac {S}{2}\lambda(v)\lambda(w)=
g(v'',w'')+ \frac {S}{2}\sum_{s=1}^3\eta_s(v')\eta_s(w') +
\frac {S}{2}\lambda(v)\lambda(w),\ v,w\in T_pK.
\end{multline}

We shall prove that this symmetric form is parallel as required, i.e, for any $\A\in TM$ and $v,w\in TK$ we have $(D_{\A}\Delta)(v,w)=0.$ From the symmetry and $Sp(1)$ invariance of $\Delta$ we have trivially for $v,w\in TK$ the identities
\begin{equation}\label{D-delta-trivial}
(D_\A\Delta)(v,w)=(D_\A\Delta)(w,v),\qquad (D_\A\Delta)(J_sv,J_sw)=(D_\A\Delta)(v,w).
\end{equation}
Furthermore, the restrictions of $\Delta(J_s ., .)$ to $TM$ are closed 2-forms on $M$. Indeed, let $\Delta_s$ be the 2-form on $M$
defined by
\begin{equation*}
\Delta_s(\A,\B)=\Delta(J_s\A,\B).
\end{equation*}
Using the identity
$(J_i\A)'=(J_i\A)''+\eta_j(\A)\xi_k-\eta_k(\A)\xi_j$ in \eqref{delta-def} we see that
\begin{equation*}
\Delta_i(\A,\B)=\omega_i(\A,\B)+\frac {S}{2}\sum_{s=1}^3\eta_s((J_i\A)')\eta_s(\B)=
(\omega_i+\frac {S}{2}\eta_j\wedge\eta_k)(\A,\B)=\frac{1}{2}d\eta_i(\A,\B),
\end{equation*}
which implies  $d\Delta_i(\A,\B,\C)=0.$ On the other hand, the
exterior derivative $d\Delta_i$ can be expressed in terms of the
covariant derivative $D\Delta_i$ through the well know formula
\begin{equation}\label{e:Delta closed cor}
d\Delta_i(\A,\B,\C)=(D_{\A}\Delta_i)(\B,\C)+(D_{\B}\Delta_i)(\C,\A)
+(D_{\C}\Delta_i)(\A,\B).
\end{equation}
Since by assumption $DJ_s=0$ we have
$(D_{\A}\Delta_s)(\B,\C)=(D_{\A}\Delta)(J_s\B,\C),$ equation \eqref{e:Delta closed cor} gives
\begin{equation}\label{D-delta}
(D_\A\Delta)(J_s\B,\C)+(D_\B\Delta)(J_s\C,\A)+(D_\C\Delta)(J_s\A,\B)=0,
\qquad \A,\B,\C\in TM.
\end{equation}
We will show that the identities \eqref{D-delta-trivial} and \eqref{D-delta} yield
$(D_{\A}\Delta)(v,w)=0$. An application of  \eqref{D-delta} gives
\begin{equation*}\begin{aligned}
& -(D_X\Delta)(Y,Z)+(D_{J_iY}\Delta)(J_iZ,X)+(D_Z\Delta)(X,Y)=0, \\
& -(D_{J_kX}\Delta)(J_kY,Z)+(D_{J_iY}\Delta)(J_iZ,X)+(D_Z\Delta)(X,Y)=0.
\end{aligned}
\end{equation*}
Therefore,
$(D_{J_sX}\Delta)(J_sY,Z)=(D_X\Delta)(Y,Z)=(D_X\Delta)(J_sY,J_sZ)$,
 which by  \eqref{D-delta-trivial},
implies $(D_{J_sX}\Delta)(Y,J_sZ)=(D_X\Delta)(Y,Z).$ It follows
$$(D_{J_sX}\Delta)(Y,Z)=-(D_X\Delta)(Y,J_sZ)=(D_X\Delta)(J_sY,Z)
=-(D_{J_sX}\Delta)(Y,Z),$$
thus
%\begin{equation}\label{D-Delta-5}
$(D_X\Delta)(Y,Z)=0.$

Another use of  \eqref{D-delta} gives
\begin{equation}\label{D-Delta-1}
(D_{\xi_i}\Delta)(J_iY,Z)+(D_Y\Delta)(Z,\xi)-(D_Z\Delta)(\tilde
N,Y)=0,
\end{equation}
which implies
\begin{equation*}
\begin{aligned}
&(D_{\xi_1}\Delta)(J_1Y,Z)=(D_{\xi_2}\Delta)(J_2Y,Z)=(D_{\xi_3}\Delta)(J_3Y,Z),
\\
&(D_{\xi_1}\Delta)(Y,J_1Z)=(D_{\xi_2}\Delta)(Y,J_2Z)=(D_{\xi_3}\Delta)(Y,J_3Z).
\end{aligned}
\end{equation*}
Therefore, we have
\begin{gather*}
(D_{\xi_i}\Delta)(Y,Z)=(D_{\xi_i}\Delta)(J_iY,J_iZ)=(D_{\xi_j}\Delta)(J_jY,J_iZ)
=
(D_{\xi_j}\Delta)(J_jY,J_jJ_kZ)=(D_{\xi_i}\Delta)(J_jY,J_iJ_kZ)\\\nonumber
=-(D_{\xi_i}\Delta)(J_jY,J_jZ)=-(D_{\xi_i}\Delta)(Y,Z),
\end{gather*}
thus
\begin{equation}\label{D-Delta-xiYZ}
(D_{\xi_s}\Delta)(Y,Z)=0.
\end{equation}
Now, a substitution  in \eqref{D-Delta-1} gives
\begin{equation}\label{D-Delta-2}
(D_Y\Delta)(Z,\xi)=(D_Z\Delta)(Y,\xi).
\end{equation}
%Furthermore,  letting $(s,\A,\B,\C)=(i,\xi_j,Y,Z)$ in the third
Invoking again  \eqref{D-delta} we find
\begin{equation*}
(D_{\xi_j}\Delta)(J_iY,Z)+(D_{Y}\Delta)(J_kZ,\xi)-(D_Z\Delta)(\tilde
N,J_kY)=0,
\end{equation*}
which together with \eqref{D-Delta-xiYZ} and \eqref{D-Delta-2} give
%\begin{equation}\label{D-Delta-3}
$(D_{J_sX}\Delta)(Y,\xi)=(D_X\Delta)(J_sY,\xi).$
%\end{equation}
In addition, it also follows \[(D_{J_kX}\Delta)(Y,\xi)=(D_X\Delta)(J_iJ_jY,\tilde
N)=(D_{J_jJ_iX}\Delta)(Y,\xi)=-(D_{J_kX}\Delta)(Y,\xi),\]
thus
%\begin{equation}\label{D-Delta-4}
$(D_X\Delta)(Y,\xi)=0$ as well.
%\end{equation}

Next, we apply  \eqref{D-delta} as follows
\begin{equation}\label{D-Delta-6}
\begin{aligned}
&-(D_{\xi_j}\Delta)(\xi_j,Z)-(D_{\xi_k}\Delta)(\xi_k,Z)+(D_Z\Delta)(\tilde
N,\tilde
N)=0,\\
&-(D_{\xi_i}\Delta)(\xi_i,Z)-(D_{\xi_j}\Delta)(\xi_j,Z)-(D_{J_jZ})\Delta(\tilde
N,\xi_j)=0.
\end{aligned}
\end{equation}
Since, $(D_{J_jZ}\Delta)(\xi,\xi_j)=0,$ the second
equation in \eqref{D-Delta-6} implies $(D_{\xi_s}\Delta)(\xi_s,X)=0,$
together with the first equation in \eqref{D-Delta-6}  give
$(D_{\xi_s}\Delta)(\xi,X)=(D_{X}\Delta)(\xi,\xi)=0.$

Finally,
from  \eqref{D-delta} we have
%\begin{equation}
$(D_{\xi_i}\Delta)(\xi,\xi)+(D_{\xi_j}\Delta)(J_k\xi,\tilde
N)-(D_{\xi_k}\Delta)(\xi,J_j\xi)=0$,
%\end{equation}
which implies
%\begin{equation}\label{D-Delta-8}
$(D_{\xi_s}\Delta)(\xi,\xi)=0.$
%\end{equation}
This completes the proof of Theorem~\ref{t:surface1}.

We record an important relation between the calibrating function and the parallel bilinear form,
\begin{equation}\label{df}
\Delta(N,\A)=-fII(N',\A)=f^2II(r,\A)=-fg(r,\A'')=df(\A'')=df(\A),
\end{equation}
which follows from Lemma~\ref{l:qc-einstein embedded} and the definition of $\Delta$, \eqref{delta-def}.

As an  application of  Theorem~\ref{t:surface1} we have the following result.
\begin{thrm}\label{surface2}
Let  $(K,G)$ be a hyper-K\"ahler manifold with Riemannian
curvature tensor $\hat R$. If $M$ is a qc-hypersurface of $K$ with
normal vector field $N$ then  we have that $\hat R_{vw}N=0$ for
all $p\in M$ and $v,w\in T_pK$. In particular,
the Riemannian curvature tensor $\hat R$ is degenerate at each
point $p$ of the hypersurface $M$.
\end{thrm}
\begin{proof}
Let $M$ be a qc-hypersurface of the hyper-K\"ahler manifold
$(K,G,J_1,J_2,J_3)$. Let $f$ and $\eta_s$ be the calibrating function and calibrated qc-structure determined in Theorem~\ref{t:surface1}, see also \eqref{def:f1}. Let us extend the second fundamental form $II$ of the embedding to
a section of the bundle $TK|_M\otimes TK|_M\rightarrow M$ by
setting $II(N,\A)=II(N,N)=II(\A,N)=0,\ \A\in TM\subset TK.$  For any $v,w\in TK$ we have
\begin{gather*}
II(v,w)=-\frac{1}{f}\Delta\left(v-G(v,N)N,w-G(w,N)\right)\\\nonumber
=-\frac{1}{f}\left\{\Delta(v,w)-G(v,N)\Delta(N,w)-G(w,N)\Delta(N,v)+G(v,N)G(w,N)\Delta(N,N)\right\}.
\end{gather*}
Using the Levi-Civita connection $D$ of the hyper-K\"ahler manifold $K$ we differentiate the above equation to
obtain
\begin{gather*}
(D_\A II)(\B,\C)=\frac{df(\A)}{f^2}\Delta(\B,\C)+\frac{1}{f}\left\{G(\B,D_\A N)df(\C)+G(\C,D_\A N)df(B)
\right\}\\
=\frac{1}{f^2}\left\{df(\A)\Delta(\B,\C)+df(\B)\Delta(\C,\A)+df(\C)\Delta(\B,\A)\right\},
\end{gather*}
which, in particular, implies
%\begin{equation}
$(D_\A II)(\B,\C)-(D_\B II)(\A,\C)=0$.
%\end{equation}
On the other hand we compute
\begin{gather*}
(D_\A II)(\B,\C)=\A(II(\B,\C))-II(D_\A \B,\C)-II(\B,D_\A\C)\\
=-\A G(D_\B N,\C)+G(D_{D_\A\B}N,\C)+G(D_\B N,D_\A\C)=-G(D_\A D_\B
N,\C)+G(D_{D_\A\B}N,\C).
\end{gather*}
For the curvature tensor $\hat R$ of $D$ we obtain
\begin{gather*}
0=(D_\A II)(\B,\C)-(D_\B II)(\A,\C)=-G(D_\A D_\B N,\C)+G(D_{D_\A\B}N,\C)
+G(D_\B D_\A N,\C)
-G(D_{D_\B\A}N,\C)\\=G(\hat R_{\A\B}N,\C),
\end{gather*}
thus $\hat R_{\A\B}N=0,\ \A,\B\in TM.$ Furthermore,
since $\hat R$ is the curvature of a hyper-K\"ahler manifold, it
has the property $\hat R(J_sv,J_sw)=\hat R(v,w),\ v,w\in TK.$
Hence,  $\hat R_{XN}N=\hat R_{J_sX,J_sN}N=0$ and $\hat
R_{J_iN,N}N=\hat R_{J_kN,J_jN}N=0,$ which completes the proof of
the theorem.
\end{proof}

\section{QC hypersurfaces in the flat hyper-K\"ahler  manifold $\Hnn$}\label{S:qc in Hn}

%\begin{rmrk}
As usual, we consider the flat hyper-K\"ahler quaternion space $\Hnn$ with its
standard quaternionic structure $\mathcal Q=\text{span}\{I,J,K\},$
determined by the  multiplication on the right by $- i$, $-j$ and $- k$, respectively. Let  $$\langle q,q'\rangle=Re\ \left(\sum_{a=1}^{n+1}q_a\overline{q'_a}\right)$$  be the flat hyper-K\"ahler metric of $\Hnn$.
If $M$ is a qc-hypersurface of $\mathbb H^{n+1}$ and
$(A,\omega,q_0)\in GL(n+1,\mathbb H)\times Sp(1)\times \Hnn$,  then the
quaternionic affine map $F:\mathbb H^{n+1}\rightarrow \mathbb
H^{n+1}$, defined by $F(x)=Ax\bar \omega +q_0$, transforms $M$ into
another qc-hypersurface $F(M)$ of $\mathbb H^{n+1}$ since
$F$  preserves the quaternion structure of $\mathbb
H^{n+1}$. In this section we will prove, as another
application of Theorem~\ref{t:surface1}, that in fact any qc-hypersurface
of $\mathbb H^{n+1}$ is congruent by the action of the quaternion affine group
$GL(n+1,\mathbb H)\times Sp(1)\rtimes \Hnn$ to one of the standard examples: the quaternionic Heisenberg group, the round sphere or the qc-hyperboloid, see Example \ref{e:Heisenberg embed}, \eqref{ex3-formula} and \eqref{ex2-formula}, respectively.
%\end{rmrk}

\subsection{Proof of Theorem \ref{t:embedded in Hn}} Let ${\iota}:M\rightarrow \Hnn$ be
a qc-embedding, with $N$ and $II$  the unit normal and the
second fundamental form of $M$. Recall, we assume $II$  to be
negative definite on the maximal $J_s$-invariant distribution $H$
of $M$. From Theorem~\ref{t:surface1}, we obtain a calibrating function
$f$ on $M$ and a parallel, $J_s$-invariant section $\Delta$ of the
bundle $(T^*K\otimes T^*K)|_M$. Clearly, since $\Delta$ is parallel,
we can find an endomorphism of the vector space $\Hnn$, which we denote again by $\Delta$, such that
\begin{equation*}
\Delta(v,w)={\langle}\Delta(v),w\rangle,\qquad v,w\in\Hnn.
\end{equation*}
By \eqref{delta-def} in Theorem~\ref{t:surface1} and \eqref{df} we have the identities
\begin{equation*}%\label{Afine-hypersurfaces-1}
\Delta\circ J_s= J_s\circ\Delta,\quad
df(A)= \langle\Delta N,{\iota_*}(A)\rangle,\quad -fII(A,B)=
\langle\Delta\big({\iota_*}(\A)\big),{\iota_*}(\B)\rangle,\quad \A,\B\in TM.
\end{equation*}
Moreover,  formula \eqref{delta-def} from the of proof the theorem shows that, depending on the constant $S$,
we have exactly one of the following three cases: (i) $\Delta$ is positive definite; (ii) $\Delta$ is of signature $(4n,4)$, or (iii) $\Delta$ is degenerate of
signature $(4n,0)$. We begin with the last case.

Assume $\Delta$ is degenerate of signature $(4n,4)$ and $\ker \Delta=\{v_0,J_1v_0,J_2v_0,J_3v_0\}$
for some unit $v_0\in \Hnn$, so that $\Hnn=\text{im} \Delta\oplus\ker \Delta$ . We define the symmetric
endomorphism $\Delta{'}$ of $\Hnn$ which is inverse to $\Delta$ on $\text{im} \Delta$ and $\ker\Delta{'}=\ker\Delta$. Thus,
we have
$$\Delta\Delta{'}v=\Delta{'}\Delta v=v-{\langle}v,v_0\rangle v_0-\sum {\langle}v,J_sv_0\rangle J_sv_0,\quad v\in\Hnn.$$
Consider the functions $h,t_m,l_m: M\rightarrow \R$, $m=0,1,2,3$, defined by
\begin{gather*}%\label{Afine-hypersurfaces-2}
h(p)={\langle}\Delta{'} N,N\rangle,\quad t_0(p)={\langle}v_0,{\iota}(p)\rangle, \quad
t_s(p)={\langle}J_sv_0,{\iota}(p)\rangle,\quad l_0(p)={\langle}v_0,N\rangle,\quad
l_s(p)={\langle}J_sv_0,N\rangle.
\end{gather*}
We compute
\begin{multline*}
dl(\A)\ =\ {\langle}v_0,dN(\A)\rangle= \frac{1}{f}{\langle}v_0,[\Delta
{\iota_*}(\A)]_{TM}\rangle = \frac{1}{f}{\langle}v_0,\Delta {\iota_*}(\A)-{\langle}\Delta {\iota_*}(\A),N\rangle N\rangle\\ =
\frac{1}{f}{\langle}v_0,\Delta {\iota_*}(\A)- df(\A)N\rangle
=\ \frac{1}{f}{\langle}\Delta v_0, {\iota_*}(\A)\rangle\ -\ \frac{df(\A)}{f}l_0 =
-\frac{df(\A)}{f}l_0,
\end{multline*}
which implies that the product $f\,l_0$ is constant on $M$,
$f\,l_0=C_0,$ $C_0\in\R$.  Similarly we have $dl_s=-l_s\frac{df}{f}$
and therefore $f\,l_s=C_s,\ s=1,2,3,$ where $C_s$ are constants.
Furthermore,
\begin{multline*}
dh(\A)= 2\langle\Delta{'}N,dN(\A)\rangle=
\frac{2}{f}\langle\Delta{'}N,\Delta
{\iota_*}(\A)-df(\A)N\rangle = \frac{2}{f}{\langle}\Delta\Delta{'}N,{\iota_*}(\A)\rangle -
\frac{2hdf(\A)}{f}\\
=
- \frac{2hdf(\A)}{f}- \frac{2}{f}\sum_{m=0}^3 l_mdt_m(\A)
=- \frac{1}{f^2}\left\{2\sum_{m=0}^3 l_mdt_m(\A)+h\,
d(f^2)(\A)\right\}.
\end{multline*}
It follows that $f^2dh+hd(f^2)=-2\sum_{m=0}^3 C_m dt_m,$ which
implies that on the manifold $M$ we have
\begin{equation}\label{f^2h}
f^2h=c+\sum_{m=0}^3c_m t_m
\end{equation}
for some constants $c,\, c_m\in\R$, $m=0,\dots,3$.  Now, consider the vector valued function $V:M\rightarrow\Hnn,$
$$V(p)=f\Delta{'}N(p)+t_0(p)v_0+\sum_{s=1}^3 t_s(p)J_sv_0,\qquad p\in M.$$
Formula \eqref{f^2h} implies ${\langle}\Delta
V,V\rangle\ =\ f^2h\ =\ c +\sum_{m=0}^3 c_m t_m.$ On the other hand, we have
\begin{gather*}
({\iota}-V)_*={\iota_*}-df\Delta{'}N-f\Delta{'}\left(\frac{1}{f}\Delta({\iota_*})-
\frac{df}{f}N\right)-dt_0 \,v_0-\sum_{s=1}^3 dt_sJ_sv_0={\iota_*}-\Delta\Delta{'}{\iota_*}-dt_0\,v_0-\sum_{s=1}^3 dt_sJ_sv_0=0.
\end{gather*}
Thus, there exists a  point $O\in\Hnn,$ such that for all
$p\in M$ we have
$$\langle\Delta\left({\iota}(p)-O\right),{\iota}(p)-O\rangle=c+\sum_{m=0}^3 c_mt_m(p).$$  The functions $t_m$ are
restrictions of the real coordinate functions in $\R^{4n+4}\cong\Hnn$
corresponding to the fixed vectors $v_0,J_sv_0,$ hence, we can find a
quaternionic affine transformation of $\Hnn$, which maps ${\iota}(M)$
into the hypersurface $|q|^2+t=0$, cf. Example~\ref{e:Heisenberg embed}.

Proceeding similarly in the cases where $\Delta$ is positive
definite or of signature $(4n,4)$ we will obtain, respectively,
\begin{equation}\label{ex3-formula}
\sum_{a=1}^{n}|q_a|^2+|p|^2=1,
\end{equation}
i.e., the $4n+3$ dimensional round sphere in $\mathbb
R^{4n+4}=\mathbb H^{n+1}$
and the hyperboloid
\begin{equation}\label{ex2-formula}
\sum_{a=1}^{n}|q_a|^2-|p|^2=-1.
\end{equation}
%Example~\ref{ex3} or Example~\ref{ex2}, respectively.
In these two cases, however, a
simpler prove is possible, by first applying  an
appropriate transformation from the linear group $GL(n+1,\mathbb H)$, which transforms
$\Delta$ into a diagonal matrix with entries $+1$ or
$-1$. Then, the transformed hypersurface will be  totally umbilical, and  one can use the corresponding
classification theorem of totally umbilical hypersurfaces in $\Hnn$  to complete the proof.

\subsection{QC hypersurfaces in the quaternionic projective space{} ${\mathbb H}P^{n+1}$}
Note that, as a quaternionic manifold, $\mathbb H^{n+1}$ is
equivalent to an open dense subset of the quaternionic projective
space $\mathbb HP^{n+1}$.  Thus, all qc-hypersurfaces of
$\mathbb H^{n+1}$ are also qc-hypersurfaces of $\mathbb HP^{n+1}$. Also, it is well known that $PGL(n+2,\mathbb H)$ is the group of quaternionic transformations of  \cite{Ku} $\mathbb
HP^{n+1}$. As a direct consequence of Theorem~\ref{t:embedded in Hn} we obtain

\begin{cor}\label{surface4} If $M$ is a connected qc hypersurface of the
quaternionic projective space $\mathbb HP^{n+1}$ then there exists
a transformation $\phi\in GL(n+2,\mathbb H)$ of $\mathbb HP^{n+1}$
which transforms $M$ into an open set $\phi(M)$ of the
qc hypersurface $M_o$, defined by

$$M_o\ = \ \{[q_1,\dots,q_{n+2}]\in \mathbb HP^{n+1} \ :\ |q_1|^2+\dots+|q_{n+1}|^2=|q_{n+2}|^2\},$$
where $[q_1,\dots,q_{n+2}]$ denote the quaternionic
homogeneous coordinates of $\mathbb HP^{n+1}$.

In particular, as an abstract qc-manifold, every
qc-hypersurface of $\mathbb HP^{n+1}$ is qc-conformally equivalent to an open set
of the quaternionic  contact (3-Sasakian) sphere $S^{4n+3}$.
\end{cor}

\begin{proof}
Theorem~\ref{t:embedded in Hn} gives a description of the
qc-hypersurfaces of $\mathbb HP^{n+1}$. Noting that the three quadrics in Theorem \ref{t:embedded in Hn} are congruent modulo
the $GL(n+2,\mathbb H)$ action on the projective space $\mathbb
HP^{n+1}$ completes the proof.
\end{proof}

{\bf Acknowledgments.} S.I. is partially supported by Contract
168/2014 with the Sofia University "St.Kl.Ohridski". I.M. is partially supported by a SoMoPro II Fellowship which is co-funded  by
the European Commission from \lq\lq{}People\rq\rq{} specific programme (Marie Curie Actions) within the EU
Seventh Framework Programme on the basis of the grant agreement REA No. 291782. It is further co-financed by the South-Moravian Region. DV was partially supported by Simons Foundation grant \#279381.

\section{Appendix}

\begin{lemma}\label{l:admis-lemma} Let $(M,H)$ be a qc manifold and
 $(\eta_s,I_s,g)$,  $(\eta'_s,I'_s,g')$ be two local qc-structures on an open set $U\subset M$ with the same horizontal space $H$.
 %, which are defined on an open set $U\subset M.$
 Then, there exist a positive function $\mathcal
F:U\rightarrow R,\ \mathcal F>0$ and a matrix-valued function
$\mathcal A=(a_{ij}):U\rightarrow SO(3)$ such that
$$(I'_1,I'_2,I'_3)=(I_1,I_2,I_3)\mathcal A,\quad
(\eta'_1,\eta'_2,\eta'_3)=\mathcal F(\eta_1,\eta_2,\eta_3)\mathcal
A,\quad g'=\mathcal Fg.$$
\end{lemma}
\begin{proof}
Since, by definition, $(\eta_1,\eta_2,\eta_3)$ and
$(\eta'_1,\eta'_2,\eta'_3)$ are frames of the bundle $L^*$, there
exists a matrix-valued function  $\mathcal A=(a_{ij}):U\rightarrow
GL(3)$ with $\eta'_s=\sum_t a_{st}\eta_t,\ s=1,2,3.$ Taking the
exterior derivative of the above equations we get
\begin{equation}\label{ad-lemma1}
(d\eta'_s)|_H=\sum_t a_{st}(d\eta_t)|_H,\ s=1,2,3.
\end{equation}
Let us pick any symmetric and positive definite section $h$ of the
bundle $H^*\otimes H^*$ which we will refer to simply as a metric
on $H.$ With respect to this metric one may consider the
restrictions of the 2-forms $(d\eta'_s)|_H,\ s=1,2,3$ to $H$ as
endomorphisms of $H,$ i.e. sections of the bundle
$End(H)=H^*\otimes H.$ Of course, this identification depends
strongly on the choice of $h$. However, it is easy to see that the
compositions of two endomorphisms of the form
$((d\eta'_s)|_H)^{-1}\circ(d\eta'_t)|_H,\ s=1,2,3$ produces
endomorphisms independent of the choice of $h$. Let us take
$(i,j,k)$ to be any cyclic permutation of $(1,2,3)$. If $h=g'$
then we get
\begin{equation}
((d\eta'_j)|_H)^{-1}\circ(d\eta'_i)|_H=I'_k.
\end{equation}
The above equation needs to hold for any choice of the
metric $h$ on $H$, in particular, also for $h=g$. Using \ref{ad-lemma1},
we conclude that
$$I'_k\ =\ ((d\eta'_j)|_H)^{-1}\circ(d\eta'_i)|_H\ \in\  \text{span}_{\mathbb R}\ \{id_H,I_1,I_2,I_3\}.$$

\noindent Note that $\text{span}_{\mathbb R}\ \{id_H,I_1,I_2,I_3\}\subset
End(H)$ is an algebra with respect to the usual composition of
endomorphisms, which is isomorphic to the algebra of the
quaternions $$\mathbb H=\text{span}_{\mathbb R}\ \{1,i,j,k\}.$$ If
an element of $\mathbb H$ has square $-1$ then this element
belongs to $Im(\mathbb H).$ Therefore $I'_s\in
Q=\text{span}\{I_1,I_2,I_3\},\ s=1,2,3$ and thus
$$\text{span}_{\mathbb R}\ \{I_1,I_2,I_3\}\ =\ \text{span}_{\mathbb R}\ \{I'_1,I'_2,I'_3\}.$$
Now, if we keep identifying $H^*\otimes  H$ with $End(H)$ by
using $h=g$, then, since the metric $g$ is $I_s$- and
$I'_s$-compatible, each of the endomorphisms $(d\eta_k')_H\in
End(H)$ anti-commutes with both $I'_i$ and $I'_j.$ This implies
that, as an endomorphism, $(d\eta_k')_H$ is proportional to
$I'_k$, and hence $g'=\mathcal Fg$ for some $\mathcal
F>0.$ The fact that $\mathcal A=(a_{ij})$ takes values in
$SO(3)$ follows from the requirement that both  $(I_1,I_2,I_3)$
and $(I'_1,I'_2,I'_3)$ satisfy the quaternionic identities.

\end{proof}

\end{document}